\documentclass{article}
\usepackage{amsthm,amsmath,amssymb,amsfonts,graphicx}
  \newtheorem{thm}{Theorem}[section]
  
  \newtheorem{prop}[thm]{Proposition}
  \newtheorem{cor}[thm]{Corollary}
  \theoremstyle{definition}
  \newtheorem{defn}[thm]{Definition}
  \newtheorem{exm}[thm]{Example}
  \newtheorem{rmk}[thm]{Remark}
  
 \newcommand\ra{\rightarrow}

\newcommand{\lex}{\,\overrightarrow{\times}\,}

 \newcommand\s{\subseteq}

\newcommand\mi{\mathcal{I}}
 \numberwithin{equation}{section}

\def\im{ \mathrm{Im} }

\begin{document}

\title{ On epicomplete $MV$-algebras }
\author{ Anatolij Dvure\v{c}enskij$^{^{1,2}}$, Omid Zahiri$^{^{3,4}}$ \\
{\small\em $^1$Mathematical Institute,  Slovak Academy of Sciences, \v Stef\'anikova 49, SK-814 73 Bratislava, Slovakia} \\
{\small\em $^2$Depart. Algebra  Geom.,  Palack\'{y} Univer. 17. listopadu 12, CZ-771 46 Olomouc, Czech Republic} \\
{\small\em  $^{3}$University of Applied Science and Technology, Tehran, Iran}\\
{\small\em  $^{4}$ Department of Mathematics, Shahid Beheshti University, G. C., Tehran, Iran}\\
{\small\tt  dvurecen@mat.savba.sk\quad   zahiri@protonmail.com} }
\date{}
\maketitle
\begin{abstract}
The aim of the paper is to study epicomplete objects in the category of $MV$-algebras.
A relation between injective $MV$-algebras and epicomplete $MV$-algebras is found, an equivalence condition for an $MV$-algebra
to be epicomplete is obtained, and it is shown that the class of divisible $MV$-algebras and the class of epicomplete $MV$-algebras are the same.
Finally, the concept of an epicompletion for $MV$-algebras is introduced, and the conditions under  which an $MV$-algebra has
an epicompletion are obtained. As a result we show that each $MV$-algebra has an epicompletion.
\end{abstract}

{\small {\it AMS Mathematics Subject Classification (2010)}:  06D35, 06F15, 06F20 }

{\small {\it Keywords:} $MV$-algebra, Epicomplete $MV$-algebra, Divisible $MV$-algebra, Injective $MV$-algebra, Epicompletion, $a$-closed $MV$-algebra. }

{\small {\it Acknowledgement:}  This work was supported by  grant VEGA No. 2/0069/16 SAV and GA\v{C}R 15-15286S.}

\section{ Introduction }
Epicomplete objects are interesting objects in each category. Many researches studied this object in the category of lattice ordered groups ($\ell$-group). Pedersen \cite{pedersen} defined the concept of an $a$-epimorphism in this category. It is an $\ell$-homomorphism which is also an epimorphism in the category of all torsion free Abelian groups.
Anderson and Conrad \cite{Anderson}  proved that each epimorphism in the category of Abelian
$\ell$-groups is an $a$-epimorphism. They showed that an Abelian $\ell$-group $G$ is epicomplete if and only if
it is divisible. They also studied epicomplete objects in some subcategories of Abelian $\ell$-groups.
In particular, they proved that epicomplete objects in the category of Abelian $o$-groups (linearly ordered groups)
with complete $o$-homomorphisms are the Hahn groups.
Darnel \cite{Darnel} continued to study these objects and showed that any completely distributive epicomplete object in
the category $\mathcal C$ of Abelian $\ell$-groups with complete $\ell$-homomorphisms is $V(\Gamma,\mathbb R)$ for some root system $\Gamma$, where $V(\Gamma,\mathbb R)$ is the set of functions $v: \Gamma \to \mathbb R$ whose support satisfies the ascending chain condition with a special order (see \cite[Prop 51.2]{2}).

He verified a new subcategory of $\mathcal{C}$ containing completely-distributive Abelian $\ell$-groups with complete $\ell$-homomorphisms and proved the converse. That is, the only epicomplete objects in this category are of the form $V(\Gamma,\mathbb R)$. Ton \cite{ton} studied epicomplete archimedean $\ell$-groups  and proved that
epicomplete objects in this category are $\ell$-isomorphic to a semicomplete subdirect sum of real groups.
Many other references can be found in \cite{Ball1,Ball2}.
Recently, Hager \cite{Hager} posed a question on the category of Archimedean $\ell$-groups  ``Does the epicompleteness imply the existence of a compatible reduced $f$-ring multiplication? ''.  His answer to this question was ``No'' and he tried to find a partial positive answer for it.

$MV$-algebras were defined by Chang \cite{Cha} as an algebraic counterpart of many-valued reasoning. The principal result of theory of $MV$-algebras is a representation theorem by Mundici \cite{Mun} saying that there is a categorical equivalence between the category of MV-algebras and the category of unital Abelian $\ell$-groups. Today theory of $MV$-algebras is very deep and has many interesting connections with other parts of mathematics with many important applications to different areas. For more details on $MV$-algebras, we recommend the monographs \cite{mundici 1, mundici 2}.

In the present paper, epicomplete objects in $\mathcal{MV}$, the category of $MV$-algebras, are studied. The concept of an $a$-extension in $\mathcal{MV}$ is introduced to obtain a condition on minimal prime ideals of an $MV$-algebra $M$ under which  $M$ is epicomplete. Some relations between injective, divisible and epicomplete $MV$-algebras are found. In the final section, we introduce a completion for an $MV$-algebra which is epicomplete and has the universal mapping property. We called it the epicompletion and we show that any $MV$-algebra has an epicompletion.

\section{ Preliminaries}
In the section, we gather some basic notions relevant to $MV$-algebras and
$\ell$-groups which will be needed in the next sections. For more details, we recommend to consult the books \cite{Anderson1, 2} for theory of $\ell$-groups and \cite{Di Nola,mundici 1, mundici 2} for $MV$-algebras.

We say that an $MV$-{\it algebra} is an algebra $(M;\oplus,',0,1)$ (and we will write simply $M=(M;\oplus,',0,1)$) of type $(2,1,0,0)$, where $(M;\oplus,0)$ is a
commutative monoid with the neutral element $0$ and, for all $x,y\in M$, we have:
\begin{enumerate}
	\item[(i)] $x''=x$;
	
	\item[(ii)] $x\oplus 1=1$; 
	
	\item[(iii)] $x\oplus (x\oplus y')'=y\oplus (y\oplus x')'$.
\end{enumerate}

In any $MV$-algebra $(M;\oplus,',0,1)$,
we can define the following further operations:
$$
x\odot y:=(x'\oplus y')',\quad x\ominus y:=(x'\oplus y)'.
$$

In addition, let $x \in M$. For any integer $n\ge 0$, we set
$$0.x=0,\quad 1.x=x,\quad n.x=(n-1).x\oplus x, \ n\ge 2,$$
and
$$x^0=1,\quad x^1 =1, \quad x^n=x^{n-1}\odot x, \ n\ge 2.$$

Moreover, the relation $x\leq y \Leftrightarrow x'\oplus y=1$ is a partial order on $M$ and $(M;\leq)$ is a lattice, where
$x\vee y=(x\ominus y)\oplus y$ and $x\wedge y=x\odot (x'\oplus y)$.
We use $\mathcal{MV}$ to denote the category of $MV$-algebras whose objects are $MV$-algebras and morphisms are $MV$-homomorphisms.
A non-empty subset $I$ of an $MV$-algebra $(M;\oplus,',0,1)$ is called an {\it ideal} of $M$ if $I$ is a down set which is closed under $\oplus$.
The set of all ideals of $M$ is denoted by $\mi(M)$.
For each ideal $I$ of $M$, the relation $\theta_I$ on $M$ defined by $(x,y)\in\theta_I$ if and only if $x\ominus y,y\ominus x\in I$ is a congruence relation on $M$, and $x/I$ and $M/I$ will denote $\{y\in M\mid (x,y)\in\theta_I\}$ and $\{x/I\mid x\in M\}$, respectively. A {\it prime} ideal is a proper ideal $I$ of $M$ such that $M/I$ is a linearly ordered $MV$-algebra, or equivalently, for all $x,y \in M$, $x\ominus y \in I$ or $y\ominus x\in I$. The set of all minimal prime ideals of $M$ is denoted by $Min(M)$. If $M_1$ is a subalgebra of an $MV$-algebra $M_2$, we write $M_1\leq M_2$.

\begin{rmk}
Let $M_1$ be a subalgebra of an $MV$-algebra $M_2$. For any ideal $I$ of $M_2$, the set $\bigcup_{x\in M_1}x/I$ is a subalgebra of $M_2$
containing $I$ which is denoted by $M_1+I$ for simplicity.
\end{rmk}

An element $a$ of an $MV$-algebra  $(M;\oplus,',0,1)$ is called a
{\it boolean} element if $a'':=(a')'=a$. The set of all boolean elements of $M$ is denoted by $B(M)$. An ideal $I$ of $M$
is called a {\it stonean} ideal if  there is a subset $S\s B(M)$ such that
$I=\downarrow S$, where $\downarrow S=\{x\in M\mid x\leq a\mbox{ for some $a\in S$ }\}$.
An element $x\in M$ is called {\it archimedean} if there is a integer $n\in\mathbb{N}$ such that $n.x$ is boolean.
An $MV$-algebra $M$ is said to be {\it hyperarchimedean} if all its elements are archimedean.  For more details about
hyperarchimedean $MV$-algebra see \cite[Chap 6]{mundici 1}

A group $(G;+,0)$ is said to be {\it partially ordered} if it is equipped with a partial order relation $\leq$ that is compatible with $+$, that is,
$a\leq b$ implies  $x+a+y\leq x+b+y$ for all $x,y\in G$. An element $x\in G$ is called {\it positive} if $0\leq x$.
A partially ordered group $(G;+,0)$ is called a {\it lattice ordered group} or simply an $\ell$-{\it group}
if $G$ with its partially order relation is a lattice.
The {\it lexicographic product} of two po-groups $(G_1;+,0)$ and $(G_2;+,0)$ is the direct product $G_1\times G_2$ endowed with the lexicographic ordering $\le$ such that $(g_1,h_1) \le (g_2,h_2)$ iff $g_1<g_2$ or $g_1=g_2$ and $h_1\le h_2$ for $(g_1,h_1), (g_2,h_2)\in G_1\times G_2$.
The lexicographic product of po-groups $G_1$ and $G_2$ is denoted by $G_1\lex G_2$.

An element $u$ of an $\ell$-group $(G;+,0)$ is called a {\it strong unit} if, for each $g\in G$, there exists $n\in \mathbb{N}$ such
that $g\leq nu$. A  couple $(G,u)$, where $G$ is an $\ell$-group and $u$ is a fixed strong unit for $G$, is said to be a {\it unital} $\ell$-{\it group}.

If $(G;+,0)$ is an Abelian $\ell$-group with strong unit $u$, then the interval $[0,u]:=\{g \in G \mid 0\leq g \leq u\}$ with the operations $x\oplus y:=(x+y)\wedge u$ and $x':=u-x$ forms
an $MV$-algebra, which is denoted by $\Gamma(G,u)=([0,u];\oplus,',0,u)$. Moreover, if $(M;\oplus,0,1)$ is an $MV$-algebra, then according to the famous theorem by Mundici, \cite{Mun}, there exists a unique (up to isomorphism)  unital Abelian $\ell$-group
$(G,u)$ with strong $u$ such that $\Gamma(G,u)$ and $(M;\oplus,0,1)$ are isomorphic (as $MV$-algebras).
Let $\mathcal{A}$ be the category of unital Abelian $\ell$-groups whose objects are unital $\ell$-groups and morphisms are unital $\ell$-group morphisms (i.e. homomorphisms of $\ell$-groups preserving fixed strong units). It is important to note that $\mathcal{MV}$ is a variety whereas $\mathcal A$ not because it is not closed under infinite products.  Then $\Gamma: \mathcal{A}\ra \mathcal{MV}$ is a functor between these categories. Moreover, there is another functor from the category of $MV$-algebras to $\mathcal{A}$ sending $M$ to a Chang $\ell$-group induced by good sequences of the $MV$-algebra $M$, which is
denoted by $\Xi:\mathcal{MV}\ra \mathcal{A}$. For more details relevant to these functors, please see \cite[Chaps 2 and 7]{mundici 1}.

\begin{thm}{\rm\cite[Thms 7.1.2, 7.1.7]{mundici 1}}\label{functor}
The composite functors $\Gamma\Xi$ and $\Xi\Gamma$ are naturally equivalent to the identity functors of
$\mathcal{MV}$ and $\mathcal{A}$, respectively. Therefore, the categories $\mathcal{A}$ and $\mathcal{MV}$ are categorically equivalent.
\end{thm}

\noindent
Next theorem states that $\mathcal{MV}$ satisfies the amalgamation property.

\begin{thm}\label{amalgam} \em{\cite[Thm 2.20]{mundici 2}} Given one-to-one homomorphisms
$A\xleftarrow{\alpha} Z\xrightarrow{\beta} B$ of $MV$-algebras, there is an $MV$-algebra $D$ together
with one-to-one homomorphisms $A\xrightarrow{\mu} D\xleftarrow{\nu} B$ such that $\mu\circ \alpha=\nu\circ \beta$.
\end{thm}

\noindent
An $MV$-algebra $(M;\oplus,',0,1)$ is called {\it divisible} if, for all $a\in M$ and all $n\in\mathbb{N}$, there exists $x\in M$ such that
\begin{itemize}
\item  $n.x=a$.
\item $a'\oplus ((n-1).x)=x'$.
\end{itemize}

Let $(M;\oplus,',0,1)$ be an $MV$-algebra and $(G,u)$ be the unital Abelian $\ell$-group
corresponding to $M$, that is $M=\Gamma(G,u)$. It can be easily seen that $M$ is divisible if and only if, for all $a\in M$ and
for all $n\in\mathbb{N}$, there exists $x\in M$ such that the group element $nx$ is defined in $M$ and $nx=a$. Moreover, $M$ is divisible if and only if
$G$ is divisible (see \cite[Lem. 2.3]{Di Nola} or \cite[Prop 2.13]{16}).
It is possible to show that if $nx=a=ny$, then $x=y$ (see \cite{Dvu2}).
If $(G(M),u)$ is the unital $\ell$-group corresponding to an $MV$-algebra $M$ and $G(M)^d$ is the divisible hull of the
$\ell$-group $G(M)$, then $G(M)^d$ is an $\ell$-group with strong unit $u$ and we use $M^d$ to denote the $MV$-algebra $\Gamma(G(M)^d,u)$.
By \cite{Dvu2}, $M^d$ is a divisible $MV$-algebra containing $M$; we call $M^d$ the {\it divisible hull} of $M$. For more details about divisible $MV$-algebras we recommend to see
\cite{Dvu2,Di Nola,lapenta}.

\begin{defn}\cite{16}\label{injective}  
	An $MV$-algebra $A$ is {\it injective} if each $MV$-homomorphism $h:C\ra A$ from every $MV$-subalgebra $C$ of an $MV$-algebra $B$ into $A$ can be
	extended to an $MV$-homomorphism from $B$ into $A$.   	
\end{defn}

\begin{defn}\cite{Dvu1}\label{summand} 
An ideal $I$ of an $MV$-algebra $M$ is called a {\it summand-ideal} if there exists an ideal $J$ of $M$ such that $\langle I\cup J\rangle=M$ and $I\cap J=\{0\}$,
where $\langle I\cup J\rangle$ is the ideal of $M$ generated by $I\cup J$.
In this case, we write $M=I\boxplus J$. The set of all summand-ideals of $M$ is denoted by $\mathfrak{Sum}(M)$. Evidently, $\{0\}, M \in \mathfrak{Sum}(M)$.
\end{defn}
\section{ Epimorphisms on class of MV-algebras}

In this section, epicomplete objects and an epimorphism in the category of $MV$-algebras are established and their properties are studied. Some
relations between epicomplete $MV$-algebras, $a$-extensions of $MV$-algebras and divisible $MV$-algebras are obtained. We show that any
injective $MV$-algebra is epicomplete. Finally, we prove that an $MV$-algebra is epicomplete if and only if it is divisible.

Recall that a morphism $f:M_1\ra M_2$ of $\mathcal{MV}$ is called an {\it epimorphism} if, for each
$MV$-algebra $M_3$ and all $MV$-homomorphisms $\alpha:M_2\ra M_3$ and $\beta:M_2\ra M_3$,
the condition $\alpha\circ f=\beta\circ f$
implies that $\alpha=\beta$.
An object $M$ of $\mathcal{MV}$ is called {\it epicomplete} if,
for each $MV$-algebra $A$ and for each epimorphism $\alpha:M\ra A$ in $\mathcal{MV}$,
we get $\alpha$ is a surjection.
An {\it epi-extension} for $M$ is an epicomplete object containing $M$ epically.
That is, the inclusion map $i:M\ra E$ is an epimorphism and $E$ is an epicomplete $MV$-algebra.

\begin{defn}\label{2.1}
	Let $M_1$ be a subalgebra of an $MV$-algebra $M_2$. Then $M_2$ is an {\it $a$-extension} of $M_1$ if the map $f: \mathcal I(M_2)\ra \mathcal I(M_1)$ defined by	 $f(J)= J\cap M_1$, $J \in \mathcal I(M_2)$, is a lattice isomorphism. An $MV$-algebra is called {\it $a$-closed} if it has no proper $a$-extension.
\end{defn}

It can be easily seen that $M_2$ is an $a$-extension for $M_1$ if and only if for all $0<y\in M_2$ there are $n\in\mathbb{N}$ and
$0<x\in M_1$ such that $y< n.x$ and $x< n.y$.

\begin{prop}\label{2.2}
	If $f:M_1\ra M_2$ is an epimorphism, then $M_2$ is an $a$-extension for $f(M_1)$.
\end{prop}

\begin{proof}
Let $I$ and $J$ be two ideals of $M_2$ such that $I\cap f(M_1)=J\cap f(M_1)$. Then by the Third Isomorphism Theorem \cite[Thm 6.18]{burris}, we get that
\begin{eqnarray}\label{R2.1}
\frac{M_2}{J}\supseteq\frac{f(M_1)+J}{J}\cong \frac{f(M_1)}{J\cap f(M_1)}=\frac{f(M_1)}{I\cap f(M_1)}\cong \frac{f(M_1)+I}{I}\s \frac{M_2}{I}.
\end{eqnarray}
Let $\alpha_I:\frac{f(M_1)}{I\cap f(M_1)}\ra \frac{M_2}{I}$ and $\alpha_J:\frac{f(M_1)}{J\cap f(M_1)}\ra \frac{M_2}{J}$  be the canonical morphisms induced from (\ref{R2.1}). Then by the amalgamation property (Theorem \ref{amalgam}), there exist an $MV$-algebra $A$ and homomorphisms $\beta_I:\frac{M_2}{I}\ra A$ and  $\beta_J:\frac{M_2}{J}\ra A$ such that  $\beta_I\circ\alpha_I=\beta_J\circ\alpha_J$. Consider the
following maps
\begin{eqnarray*}
\mu_I:M_2\xrightarrow{\pi_I}\frac{M_2}{I}\xrightarrow{\beta_I}A,\quad
\mu_J:M_2\xrightarrow{\pi_J}\frac{M_2}{J}\xrightarrow{\beta_J}A,
\end{eqnarray*}
where $\pi_I$ and $\pi_J$ are the natural projection homomorphisms. For all $x\in M_1$,
\begin{eqnarray*}
\mu_I(f(x))=\beta_I(\frac{f(x)}{I})=\beta_I(\alpha_I(\frac{f(x)}{I\cap f(M_1)}))=\beta_J(\alpha_J(\frac{f(x)}{J\cap f(M_1)}))=\beta_J(\frac{f(x)}{J})=\mu_J(f(x)).
\end{eqnarray*}
It follows that $\mu_I\circ f=\mu_J\circ f$ and so by the assumption $\mu_I=\mu_J$, which implies that $I=J$.
Therefore, $M_2$ is an $a$-extension for $f(M_1)$.
\end{proof}

\noindent
The next theorem helps us to prove Corollaries \ref{2.3.0} and \ref{2.3.1}.

\begin{thm}\label{2.3}
Let $M_1$ be a subalgebra of an $MV$-algebra $(M_2;\oplus,',0,1)$ such that $M_2$ is an $a$-extension of $M_1$ and $M_1+I=M_2$ for all $I\in Min(M_2)$. Then $M_1=M_2$.
\end{thm}

\begin{proof}
Choose $b\in M_2 \setminus M_1$ and set $S:=\{x\ominus b\mid x\in M_1,~ x\vee b\in M_1~ and ~ x\ominus b>0 \}$.
 Clearly, $S\neq\emptyset$ and  $0\notin S$.
First we show that $S$ is closed under $\wedge$.
Let $x,y\in M_1$ be such that $x\ominus b,y\ominus b\in S$. Then $x\vee b,y\vee b\in M_1$ and $x\ominus b,y\ominus b>0$. We claim that $(x\wedge y)\ominus b>0$. From \cite[Props 1.15, 1.16, 1.21, 1.22]{georgescu}
it follows that $(x\wedge y)\ominus b=(x\ominus b)\wedge (y\ominus b)$.

If $(x\wedge y)\ominus b=0$, then
\begin{eqnarray*}
x\wedge y\leq b\Rightarrow (x\wedge y)\vee b=b\Rightarrow (x\vee b)\wedge (y\vee b)=b
\end{eqnarray*}
but $(x\vee b)\wedge (y\vee b)\in M_1$(since $x\vee b,y\vee b\in M_1$), which is a contradiction. So $0<(x\wedge y)\ominus b$.
Similarly, we can show that $(x\wedge y)\vee b\in M_1$.
Hence $(x\ominus b)\wedge (y\ominus b)\in S$. It follows that
there is a proper lattice filter of $M_1$ containing $S$ which implies that  there exists a maximal lattice filter of $M_1$ containing
$S$, say $\overline{S}$, whence $\overline{S}=M_1\setminus P$ for some minimal prime lattice filter $P$ of $M_1$.
By \cite[Cor 6.1.4]{mundici 1}, $P$ is a minimal prime filter of $M_1$, and so there exists $Q\in Min(M_2)$ such that $P=Q\cap M_1$.
By the assumption and by the Third Isomorphism Theorem,
$$
\frac{M_1}{Q\cap M_1}\cong \frac{M_1+Q}{Q}\cong \dfrac{M_2}{Q}.
$$
Then there exists $a\in M_1$ such that  $a/Q=b/Q$, so $b\ominus a,a\ominus b\in Q$. Clearly, $(b\ominus a)\vee (a\ominus b)\neq 0$
(otherwise, $b=a\in M_1$ which is a contradiction).

\noindent
(i) If  $a\ominus b=0$, then $b\ominus a>0$. Let $0<b\ominus a=t\in Q$. Then there are $n\in\mathbb{N}$ and $z\in M_1$
such that $t<n.z$ and $z<n.t$, so $z,n.z\in Q$ which implies that $n.z\in Q\cap M_1$. From $b\ominus a\leq n.z$, we have
$b\leq a\oplus n.z$. Clearly, $b<a\oplus n.z$ (since $a\oplus n.z\in M_1$). Thus $(a\oplus n.z)\ominus b>0$ and
$(a\oplus n.z)\vee b=a\oplus n.z\in M_1$ and hence by definition
\begin{eqnarray}
\label{R 2.2}(a\oplus n.z)\ominus b\in S.
\end{eqnarray}
On the other hand, in view of
$\frac{(a\oplus n.z)\ominus b}{Q}=(\frac{a}{Q}\oplus \frac{n.z}{Q})\ominus \frac{b}{Q}=
\frac{a}{Q}\ominus \frac{b}{Q}=\frac{0}{Q}$, we get 
\begin{eqnarray}
\label{R 2.3}(a\oplus n.z)\ominus b\in Q.
\end{eqnarray}
From relations (\ref{R 2.2}) and (\ref{R 2.3}) it follows that $(a\oplus n.z)\ominus b\in S\cap Q$ which is a contradiction.

(ii) If $b\ominus a=0$, then $b\leq a$ and $a\ominus b>0$, so $a\ominus b\in S\cap Q$
(note that $b\vee a=a\in M_1$) which is a contradiction.

(iii) If $b\ominus a>0$ and $a\ominus b>0$, then $a\ominus b=t\in Q$, so similarly to (i)
there are $n\in\mathbb{N}$ and $z\in M_1$  such that $a\ominus b<n.z\in Q\cap M_1$.
It follows that $a\ominus n.z\leq b$. Since $a\ominus n.z\in M_1$, we have $a\ominus n.z<b$. Hence, $(a\ominus n.z)\ominus b=0$,
$\frac{a\ominus n.z}{Q\cap M_1}=\frac{a}{Q\cap M_1}$ and $\frac{a\ominus n.z}{P}=\frac{b}{P}$.
Now, we return to (i) and replace $a$ with $a\ominus n.z$. Then we get another contradiction.
Therefore, the assumption was incorrect and there is no $b\in M_2\setminus M_1$. That is, $M_2=M_1$.
\end{proof}

\begin{cor}\label{2.3.0}
An $MV$-algebra $(A;\oplus,',0,1)$ is epicomplete if and only if for each epimorphism $f:A\ra B$, we have $\im(f)+I=B$ for all $I\in Min(B)$.
\end{cor}

\begin{proof}
The proof is straightforward by Proposition \ref{2.2} and Theorem \ref{2.3}.
\end{proof}

\begin{cor}\label{2.3.1}
An $MV$-algebra $(M;\oplus,',0,1)$ is divisible if and only if $M/P$ is divisible for each $P\in Min(M)$.
\end{cor}

\begin{proof}
Consider the divisible hull  $M^d$ of the $MV$-algebra $M$. For each $n\in \mathbb{N}$ and
each $y\in M^d$, there exists $x\in X$ such that $n.x=y$. It entails that $y<(n+1).x$ and
$x<(n+1).y$. So, $M^d$ is an $a$-extension for $M$. It follows that $Min(M)=\{P\cap M\mid P\in Min(M^d)\}$.
Moreover, for each $P\in Min(M^d)$, we have
$\frac{M}{P\cap M}=\frac{P+M}{P}\subseteq \frac{M^d}{P}$ and so by the assumption $\frac{P+M}{P}$ is
divisible. It follows that $\frac{P+M}{P}=\frac{M^d}{P}$ (since $\frac{M^d}{P}$ is a divisible extension of
$\frac{P+M}{P}$), hence $P+M=M^d$. Now, by Theorem \ref{2.3},  we conclude that $M=M^d$. Therefore,
$M$ is divisible. The proof of the converse is straightforward.
\end{proof}

In Theorem \ref{2.5},
we show a condition  under which an $MV$-algebra is $a$-closed.

\begin{rmk}\label{2.4}
If $M_2$ is an $a$-extension for $MV$-algebra $M_1$, then for all $I\in I(M_2)$, the $MV$-algebra $\frac{M_2}{I}$ is an
$a$-extension for the $MV$-algebra $\frac{M_1+I}{I}$.
\end{rmk}

\begin{defn}\label{2.6}
	An ideal $I$ of an $MV$-algebra $(M;\oplus,',0,1)$ is called an $a$-ideal if $\frac{M}{I}$ is an $a$-closed $MV$-algebra. Clearly, $M$ is an $a$-closed ideal of $M$. Moreover, $M$ is $a$-closed if and only if $\{0\}$ is an $a$-closed ideal.
\end{defn}

\begin{thm}\label{2.5}
If every minimal prime ideal of an $MV$-algebra $(M;\oplus,',0,1)$ is $a$-closed, then $M$ is $a$-closed.
\end{thm}

\begin{proof}
Let $A$ be an $a$-extension for $M$. For each $P\in Min(A)$, we have $\frac{M}{P\cap M}\cong \frac{M+P}{P}\xrightarrow{\s} \frac{A}{P}$. By the above remark,
$\frac{A}{P}$ is an $a$-extension for $\frac{M+P}{P}$. Since $\frac{M+P}{P}$ is $a$-closed, then $\frac{M+P}{P}=\frac{A}{P}$. Now, from Theorem \ref{2.3},
it follows that $M=A$.
\end{proof}

Clearly, the converse of Theorem \ref{2.5} is true, when $M$ is linearly ordered. Indeed, if $M$ is a chain, $\{0\}$ is the only minimal
prime ideal of $M$ and so $M\cong \frac{M}{\{0\}}$ is $a$-closed. In the following proposition and corollary,
we try to find a better condition under which the converse of Theorem \ref{2.5} is true.

\begin{prop}\label{2.7}
If $I$ is a summand ideal of an $a$-closed $MV$-algebra $(M;\oplus,',0,1)$, then $\frac{M}{I}$ is $a$-closed.
\end{prop}

\begin{proof}
Let $M$ be an $a$-closed $MV$-algebra and $I$ be a summand ideal of $M$. By \cite[Cor 3.5]{Dvu1},
there exists $a\in B(M)$ such that $I=\downarrow a$, $I^{\bot}=\downarrow a'$ and $M=\downarrow a\oplus \downarrow a' :=\{x\oplus y \mid x\in \downarrow a, \, y \in \downarrow a'\}$. Moreover, for each
$x\in M$, there are $x_1\leq a$ and $x_2\leq a'$ such that $x=x_1\oplus x_2$ and so $x/I=x_2/I$. Hence for each $x,y\in M$,
\begin{eqnarray*}
x/I=y/I&\Leftrightarrow & x_2/I=y_2/I\Leftrightarrow x_2\ominus y_2,y_2\ominus x_2\in I \\
&\Rightarrow & x_2\ominus y_2\leq x_2\in I^{\bot},\  y_2\ominus x_2\leq y_2\in I^{\bot}\\
&\Rightarrow & x_2\ominus y_2,y_2\ominus x_2\in I\cap I^{\bot}=\{0\} \Rightarrow x_2=y_2.
\end{eqnarray*}
That is, $\frac{M}{I}=\{x/I|\ x\in I^{\bot}\}$. Now, we define the operations $\boxplus$ and $*$ on $\downarrow a'$ by $x\boxplus y=x\oplus y$ and $x^{*}=t$, where $t$ is the second component of $x'$ in $\downarrow a\oplus \downarrow a'$. It can be easily seen that $I^\bot$ with these operations and $0$ and $a'$ as the least and greatest elements, respectively, is an $MV$-algebra. Moreover, $\frac{M}{I}\cong I^{\bot}$. Similarly, $\frac{M}{I^\bot}$ is an $MV$-algebra and $I=\downarrow a\cong \frac{M}{I^\bot}$. Now, let $A$ be an $a$-extension for the $MV$-algebra $\frac{M}{I}$. Then
\begin{eqnarray*}
\phi: M\xrightarrow{x\mapsto x_1\oplus x_2}\downarrow a\oplus \downarrow a'\xrightarrow{x\oplus y\mapsto (x/I,y/I^\bot)}  \frac{M}{I}\times \frac{M}{I^\bot}\xrightarrow{\s} A\times \frac{M}{I^\bot}.
\end{eqnarray*}
(1) Since $M\cong \frac{M}{I} \times \frac{M}{I^\bot}$, then $\frac{M}{I}\times \frac{M}{I^\bot}$ is $a$-closed.

\noindent
(2)  $A\times \frac{M}{I^\bot}$ is an $a$-extension for $\frac{M}{I}\times \frac{M}{I^\bot}$.

It follows that $\frac{M}{I}\times \frac{M}{I^\bot}=A\times \frac{M}{I^\bot}$ and so $A=\frac{M}{I^\bot}$.
In a similar way, we can show that $M/I$ is $a$-closed.
\end{proof}

\begin{cor}\label{2.8}
Let $(M;\oplus,',0,1)$ be a closed hyperarchimedean $MV$-algebra.
By {\rm \cite[Thm 6.3.2]{mundici 1}} every principal ideal of $M$ is a stonean ideal.
Hence by {\rm \cite[Cor 3.5(iii)]{Dvu1}}, we get that every principal ideal of $M$ is a summand ideal of $M$ and so $\frac{M}{I}$ is $a$-closed for each principal ideal $I$ of $M$. That is, each principal ideal of $M$ is an $a$-ideal.
\end{cor}

We note that an $MV$-algebra $M$ is {\it simple} if $\mathcal I(M)=\{\{0\},M\}$.

\begin{exm}\label{2.9}
(1)  Consider the $MV$-algebra of the real unit interval $A=[0,1]$.  Let $B$ be an $a$-extension for $A$. Then $B$ is a simple $MV$-algebra
(since $[0,1]$ is simple and $\mathcal{I}(A)\cong \mathcal{I}(B)$).  By \cite[Thm 3.5.1]{mundici 1}, $B$ is isomorphic to a subalgebra
of $[0,1]$. Let $f:B\ra [0,1]$ be a one-to-one $MV$-algebra homomorphism. Let $h:[0,1]\ra [0,1]$ be the map $h=f\circ i$,
where $i:A\ra B$ is the inclusion map. By Theorem \ref{functor}, $H:=\Xi(h):(\mathbb{R},1)\ra (\mathbb{R},1)$ is
the unital $\ell$-group homomorphism relevant to $h$. Since $H(1)=1$, it can be easily seen that $H$ is the identity map on $\mathbb{R}$.
Indeed, clearly $H(x)=x$ for all $x\in\mathbb{Q}$. Let $x_0\in \mathbb{R}$. For all $\epsilon>0$, let $\delta>0$ be an element of $\mathbb{Q}$ such that $\delta<\epsilon$. Then for  all $x\in\mathbb{R}$ with
$|x-x_0|\leq \delta$, we have $0\leq x-x_0\leq\delta$ or $0\leq x_0-x\leq \delta$. In each case, we have
$|H(x)-H(x_0)|=|H(x-x_0)|\leq H(\delta)=\delta\leq\epsilon$ (note that $H$ is order preserving). That is, $H$ is a continuous map. So clearly,  $H=Id_{\mathbb{R}}$ (since $\mathbb{Q}$ is dense in $\mathbb{R}$). It follows that $h=Id_A$, hence
$f$ is an isomorphism, and  $f$ is the identity map. Thus $B=[0,1]$. Therefore, $A$ is $a$-closed.

(2) Let $A$ be a subalgebra of the real interval $MV$-algebra $[0,1]=\Gamma(\mathbb R,1)$. Then $A$ is $a$-closed if and only if $A=[0,1]$. Indeed, one direction was proved in the forgoing case (1). Now let $A$ be a proper subalgebra of $[0,1]$. Then the $MV$-algebra $[0,1]$ is an $a$-extension of $A$ such that $A\ne [0,1]$.

(3) {\it A simple $MV$-algebra $A$ is $a$-closed if and only if $A$ is isomorphic to the $MV$-algebra $[0,1]$.}

Let $B$ be any $a$-extension of $A$. Then $B$ has to be a simple $MV$-algebra, and by \cite[Thm 3.5.1]{mundici 1}, $B$ is isomorphic to some subalgebra $\overline{B}$ of the $MV$-algebra $[0,1]$. Let $f:B\to \overline{B}\subseteq [0,1]$ be such an $MV$-isomorphism. Then $\overline{A}:=f(A)$ is a subalgebra of both $\overline{B}$ and $[0,1]$. Let $g:A\to [0,1]$ be any injective  $MV$-homomorphism. By \cite[Cor 7.2.6]{mundici 1}, $g(A)=f(A)$. Hence, if $A$ is isomorphic to $[0,1]$, then $[0,1]=\overline{A} \subseteq \overline{B}\subseteq [0,1]$ which yields $A=B$ and $A$ is $a$-closed.

Let $A$ be not isomorphic to $[0,1]$ and let $A=\Gamma(G,u)$ for some unital Abelian linearly ordered group $(G,u)$, in addition, $G$ is Archimedean.  There is a subgroup $R$ of the group $\mathbb R$ of reals, $1 \in R$, $R \ne \mathbb R$, such that $A\cong \Gamma(R,1)\varsubsetneq \Gamma(\mathbb R,1)$. By \cite[Lem 4.21]{Goo}, $R$ is either cyclic or dense in $\mathbb R$. If $R$ is cyclic, then $R =\frac{1}{n}\mathbb Z$, where $\mathbb Z$ is the group of integers and $n\ge 1$ is an integer. Then $\frac{1}{n}\mathbb Z \varsubsetneq \frac{1}{2n}\mathbb Z$, which yields $\Gamma(\frac{1}{n}\mathbb Z,1) \varsubsetneq \Gamma(\frac{1}{2n}\mathbb Z,1)$.

Let $G^d$ be the divisible hull of $G$. Define $\frac{1}{2}G:=\{\frac{1}{2}g \mid g \in G\}$. Then $(G,u) \varsubsetneq (\frac{1}{2}G,u) \subseteq (G^d,u)$. If we set $B=\Gamma(\frac{1}{2}G,u)$, then $A \varsubsetneq B\cong \Gamma(\frac{1}{2n}\mathbb Z,1)$, $B$ is simple (because $\Gamma(\frac{1}{2n}\mathbb Z,1)$ is simple), and $B$ is a proper $a$-extension of $A$, so that $A$ is not $a$-closed.

Now let $R$ be not cyclic, so it is dense in $\mathbb R$. Since $G$ is Archimedean, it has the Dedekind--MacNeille completion $G^\diamondsuit$ which is a conditionally complete $\ell$-group such that, for every $h\in G^\diamondsuit$, $h = \bigvee \{g\in G \mid g \le h\}$, and if $H$ is any complete $\ell$-group with these properties, there exists a unique $\ell$-isomorphism $\phi:G^\diamondsuit \to H$ such that $\phi(g)=g$ for all $g \in G$, see \cite[Thm 1.1]{CoMc}. Whereas the Dedekind--MacNeille completion of $R$ is the whole group $\mathbb R$, we have $G \hookrightarrow G^\diamondsuit\cong \mathbb R$. Hence, if we set $B=\Gamma(G^\diamondsuit,u)$, then $B$ is a simple MV-algebra containing $A$ as its subalgebra, and $B$ is a proper $a$-extension, proving $A$ is not $a$-closed.
\end{exm}

\begin{thm}\label{2.10}
Every injective $MV$-algebra is epicomplete.
\end{thm}
\begin{proof}
Let $(E;\oplus,',0,1)$ be an $MV$-algebra, $(M;\oplus,',0,1)$ be an injective $MV$-algebra and $f:M\ra E$ be an epimorphism.
Consider the maps given in  Figure \ref{Fig1}.
\setlength{\unitlength}{1mm}
\begin{figure}[!ht]
\begin{center}
\begin{picture}(40,27)
\put(12,25){\vector(2,0){12}}
\put(8,22){\vector(0,-1){12}}
\put(30,24){\vector(1,0){12}}
\put(30,26){\vector(1,0){12}}
\put(5,6){\makebox(4,2){{ $M$}}}
\put(5,24){\makebox(4,2){{ $M$}}}
\put(25,24){\makebox(4,2){{ $E$}}}
\put(42,24){\makebox(4,2){{ $E$}}}
\put(15,27){\makebox(4,2){{ $f$}}}
\put(1,15){\makebox(4,2){{$Id_{_M}$}}}
\put(33,28){\makebox(4,2){{$Id_{_E}$}}}
\put(33,20){\makebox(4,2){{$g$}}}
\end{picture}
\caption{\label{Fig1} }
\end{center}
\end{figure}\\
Then there exists a homomorphism $h:E\ra M$ such that  $h\circ f=Id_{_M}$.
Let $g:E\ra E$ be defined by $g=f\circ h$. Then for each $x\in M$,
$Id_{_E}\circ f(x)=f(x)=f(h\circ f(x))=(f\circ h)\circ f(x)=g\circ f(x)$ which implies that $g=Id_{_E}$ (see the latter diagram). Therefore,
$f$ is onto. That is, $M$ is epicomplete.
\end{proof}

It is well known that divisible and complete $MV$-algebras coincide with  injective $MV$-algebras (see \cite[Thm 1]{lacava} and \cite[Thm 2.14]{16}). So we have the following result.

\begin{cor}\label{2.11}
Every complete and divisible $MV$-algebra is epicomplete.
\end{cor}

Let $(M;\oplus,',0,1)$ be an $MV$-algebra and $(G,u)$ be a unital $\ell$-group such that $M=\Gamma(G,u)$.
Set $M^{d}=\Gamma(G^d,u)$, where $G^d$ is the divisible hull of $G$. Let $i: M\ra M^d$ be the inclusion map. Then
$i$ is an epimorphism. Indeed, if $A$ is another $MV$-algebra and $\alpha,\beta:M^d\ra A$ be $MV$-homomorphisms
such that $\alpha\circ i=\beta\circ i$, then by Theorem \ref{functor}, we have the following homomorphisms in $\mathcal{A}$
$$\Xi(i):(G,u)\mapsto (G^d,u),\quad \Xi(\alpha),\Xi(\beta):\Xi(M^d,u)\mapsto (\Xi(A),v),$$
where $v$ is a strong unit of $\Xi(A)$ such that $\Gamma(\Xi(A),v)=A$.
Since $\Xi$ is a functor from $\mathcal{MV}$ to $\mathcal{A}$, then we have
$\Xi(\alpha)\circ \Xi(i)=\Xi(\alpha\circ i)=\Xi(\beta\circ i)=\Xi(\beta)\circ \Xi(i)$.
By \cite[Sec 2]{Anderson}, we know that the inclusion map $\Xi(i):(G,u)\ra (G^d,u)$ is
an epimorphism, so $\Xi(\alpha)=\Xi(\beta)$, which implies that $\alpha=\beta$.
Therefore,  $i: M\ra M^d$ is an epimorphism in $\mathcal{MV}$ and so we have the following theorem:

\begin{thm}\label{2.10.0}
Epicomplete $MV$-algebras are divisible.
\end{thm}

\begin{thm}\label{2.12}
Let $(M;\oplus,',0,1)$ be an $MV$-algebra. Then $M$ is epicomplete if and only if each epimorphism of $M$ into a linearly ordered $MV$-algebra is onto.
\end{thm}

\begin{proof}
Suppose that each epimorphism of $M$ into a linearly ordered $MV$-algebra $H$ is onto.
If $f:M\ra H$ is an epimorphism, then for each $P\in Min(H)$, the map $M\xrightarrow{f}H\xrightarrow{\pi_P}\frac{H}{P}$ is
an epimorphism, where $\pi_P$ is the natural homomorphism. Since $\frac{H}{P}$ is a linearly ordered $MV$-algebra,
then by the assumption, $\pi_{P}\circ f$ is onto and so $\frac{f(M)}{P}=\frac{H}{P}$ or equivalently, $\bigcup_{x\in M}f(x)/P=H$.
Hence $f(M)+P=H$ for all $P\in Min(H)$. By Proposition \ref{2.2}, we know that $H$ is an $a$-extension for $f(M)$ and so
by Theorem \ref{2.3}, $f(M)=H$. Therefore, $M$ is epicomplete. The proof of the other direction is clear.
\end{proof}

\begin{defn}\label{Alg1}
Let $(M_1;\oplus,',0,1)$, $(M_2;\oplus,',0,1)$ and $(M_3;\oplus,',0,1)$ be $MV$-algebras such that $M_1\leq M_2$ and $M_1\leq M_3$. An element
$b\in M_2$ is {\it equivalent}  to an element $c\in M_3$ if there exists an isomorphism $f$ between $\langle M_1\cup\{b\}\rangle_{_{M_2}}$ and
$\langle M_1\cup\{c\}\rangle_{_{M_3}}$ such that $f|_{_{M_1}}=Id_{_{M_1}}$, where $\langle M_1\cup\{b\}\rangle_{_{M_2}}$ is the $MV$-subalgebra of
$M_2$ generated by $M_1\cup\{b\}$. The element $b$ is called {\it algebraic over} $M_1$ if no extension of $M_1$
contains two elements equivalent to $b$. Moreover, $M_2$ is an {\it algebraic extension} of $M_1$ if every element of
$M_2$ is algebraic over $M_1$.
\end{defn}

\begin{prop}\label{Alg2}
Let $(M_1,\oplus,',0,1)$ and $(M_2,\oplus,',0,1)$ be two $MV$-algebras such that $M_1\leq M_2$. Then $y\in M_2$ is
algebraic over $M_1$ if and only if the inclusion map $i:M_1\ra \langle M_1\cup \{y\}\rangle_{M_2}$ is an epimorphism.
\end{prop}

\begin{proof}
Let $y\in M_2$ be algebraic over $M_1$. If $i:M_1\ra \langle M_1\cup \{y\}\rangle$ is not an epimorphism, then
there exist an $MV$-algebra $M_3$ and two homomorphisms $\alpha,\beta:\langle M_1\cup \{y\}\rangle\ra M_3$
such that $\alpha\circ i=\beta\circ i$ and $\alpha\neq \beta$. Then $\alpha(y)\neq \beta(y)$ (otherwise, $\alpha=\beta$). Consider
the maps $\lambda,\mu:\langle M_1\cup\{y\}\rangle_{_{M_2}}\ra M_3\times \langle M_1\cup\{y\}\rangle_{_{M_2}}$ defined by
$\lambda(x)=(\alpha(x),x)$ and $\mu(x)=(\beta(x),x)$ for all $x\in \langle M_1\cup\{y\}\rangle_{_{M_2}}$.
Clearly, $\lambda$ and $\mu$ are one-to-one homomorphisms.
We have $\lambda(M_1)=\{(\alpha(x),x)\mid x\in M_1 \}=\mu(M_1)$ and $M_1\cong \lambda(M_1)\cong \mu(M_1)$.
Set $M=M_3\times \langle M_1\cup\{y\}\rangle_{_{M_2}}$.
We identify $M_1$ with its image in $M$ under $\lambda$. Then $M$ is an extension for $M_1$.  Since
$\lambda:\langle M_1\cup \{y\}\rangle_{_{M_2}}\ra \langle M_1\cup \{(\alpha(y),y)\}\rangle_{_M}$ and
$\mu:\langle M_1\cup \{y\}\rangle_{_{M_2}}\ra \langle M_1\cup \{(\beta(y),y)\}\rangle_{_M}$ are isomorphisms, then
$y$ is equivalent to $(\alpha(y),y)$ and $(\beta(y),y)$, which is a contradiction. Therefore,
$i:M_1\ra \langle M_1\cup \{y\}\rangle$ is an epimorphism. The proof of the converse is clear.
\end{proof}

Corollary \ref{2.11} showed that complete and divisible $MV$-algebras are epicomplete. In the sequel,
we will use the same argument as in the proof of \cite[Thm 2.1]{Anderson} with a little modification to show that
every divisible $MV$-algebra is epicomplete.

\begin{thm}\label{2.13}
	Let $(M;\oplus,',0,1)$ be an $MV$-algebra.
	\begin{itemize}
		\item[\em{(i)}]  If $(L;\oplus,',0,1)$ is a linearly ordered $MV$-algebra and $f:M\ra L$ is an epimorphism, then $L\s f(M)^d$.
		\item[\em{(ii)}]  If $M$ is divisible, then it is epicomplete.
	\end{itemize}
\end{thm}

\begin{proof}
(i) Let $(G,u)$ and $(H,v)$ be the unital $\ell$-groups such that $\Gamma(G,u)=M$ and $\Gamma(H,v)=L$.
By Theorem \ref{functor}, $\Xi(f):G\ra H$ is a unital $\ell$-group homomorphism.
Let $K=B/Im(\Xi(f))$ be the torsion subgroup of $H/Im(\Xi(f))$
(clearly $B$ is a subgroup of $G$ and $x\in B \Leftrightarrow nx\in Im(\Xi(f))$ for some $n\in\mathbb{N}$).
Since $H/B\cong \frac{H/Im(\Xi(f))}{B/Im(\Xi(f))}$ is torsion free,
by \cite[Prop 1.1.7]{Anderson1}, $H/B$ admits a linearly ordered group structure.
By \cite[Exm 3]{Glass}, $(H/B)\lex H$ is an $\ell$-group.
Since $v$ is a strong unit of an $\ell$-group $H$,
for each $(x+B,y)\in (H/B)\lex H$, there exists $n\in\mathbb{N}$ such that $x,y<nv$ and so $(x+B,y)<n(v+B,v)$. It follows that
$w:=(v+B,v)$ is a strong unit for the $\ell$-group $(H/B)\lex H$. Let $\alpha,\beta:L\ra \Gamma((H/B)\lex H,w)$ be
defined by $\alpha(x)=(B,x)$ and $\beta(x)=(x+B,x)$ for all $x\in L$ (both of them are $MV$-homomorphisms). We have
$\alpha\circ f(x)=(B,f(x))=\beta\circ f(x)$ for all $x\in L$. Since $f$ is an epimorphism, then $\alpha=\beta$, hence
for all $x\in L=\Gamma(H,v)$, we have $x\in B$ and so there is $n\in\mathbb{N}$ such that $nx\in Im(\Xi(f))$.
That is, $x$ belongs to the divisible hull of $Im(\Xi(f))$. Since $x\leq v$, then $x\in \Gamma((Im(\Xi(f)))^d,v)=f(M)^d$.
Therefore, $L\s f(M)^d$.

(ii)  Let $M$ be divisible. We use Theorem \ref{2.12} to show that $M$ is epicomplete.
Let $(A;\oplus,',0,1)$ be a linearly ordered $MV$-algebra and $f:M\ra A$ be an epimorphism into $A$.
By (i), $f(M)\s L\s (f(M))^d$. Clearly, $f(M)$ is divisible, so $f(M)=L=(f(M))^d$. It follows from Theorem \ref{2.12}
that $M$ is epicomplete.
\end{proof}

Concerning the proof of (i) in the latter theorem, we note that since $f:M\ra f(M)$ is onto, by \cite[Lem. 7.2.1]{Mun}, $\Xi(f):\Xi(M)\ra \Xi(f(M))$ is onto ($\Xi(M)=G$),
so $Im(\Xi(f))=\Xi(f(M))$. It follows that $f(M)^d=\Gamma((\Xi(f(M)))^d,v)=\Gamma((Im(\Xi(f)))^d,v)$.

\section{Epicompletion of  $MV$-algebras} 

The main purpose of the section is to introduce an epicompletion for an $MV$-algebra and to discuss about conditions when an $MV$-algebra has an epicompletion. First we introduce an epicompletion in $\mathcal{MV}$.
An epicompletion for an $MV$-algebra $A$ is an $MV$-algebra
$M$ epically containing $A$ with the universal property.
Then we use some results of the second section and prove that any $MV$-algebra has an epicompletion.
Indeed, the epicompletion of $A$ is $A^d$.

\begin{defn}\label{4.1}
	Let $(A;\oplus,',0,1)$ be an $MV$-algebra.
	\begin{itemize}
		\item[{\rm (i)}]  A pair $(\overline{A},\alpha)$, where $\overline{A}$ is an $MV$-algebra and $\alpha:A\ra \overline{A}$ is a
	one-to-one epimorphism (epiembedding for short), is called an {\it $e$-extension} for $A$.
	
	\item[{\rm (ii)}] An $e$-extension $(E,\alpha)$ for $A$ is called an {\it epicompletion} for $A$ if, for each epimorphism $f:A\ra B$,
	there is an $e$-extension $(\overline{B},\beta)$ for $B$ and a surjective homomorphism ${\bf f}: E\ra \overline{B}$ such that
	$\beta\circ f={\bf f}\circ \alpha$, or equivalently, the diagram given by Figure \ref{dig 4} commutes.
	\begin{figure}[!ht]
	\begin{eqnarray*}
	\begin{array}{ll}
	\begin{tabular}{ccccc}
	& $A$ & $\xrightarrow{ \quad  f\circ \alpha \quad  }$ & $B$\\
	 \scriptsize$\alpha$\!\!\!\!\!\!\!\!\!\! & $\Big\downarrow$ & & $\Big\downarrow$ &\!\!\!\!\!\!\!\!\!\! \scriptsize$\beta$ \\
	& $E$& $\xrightarrow{\quad {\bf f} \quad }$ & $\overline{B}$ &
	\end{tabular}
	\end{array}
	\end{eqnarray*}
	\caption{\label{dig 4}}
	\end{figure}
\end{itemize}
\end{defn}

\begin{prop}\label{4.2}
Let $(A;\oplus,',0,1)$ be an $MV$-algebra.
\begin{itemize}
\item[{\em (i)}] Each epicompletion of $A$ is epicomplete.
\item[{\em (ii)}] If $A$ has an epicompletion, then it is unique up to isomorphism.
\end{itemize}
\end{prop}

\begin{proof}
(i) Let $(\overline{A},\alpha)$ be an epicompletion for $A$ and $f:\overline{A}\ra B$ be an epimorphism.
Then $f\circ \alpha:A\ra B$ is an epimorphism and so there exists an $e$-extension $(\overline{B},\beta)$ for $B$
and an onto morphism $h:\overline{A}\ra \overline{B}$ such that the diagram in Figure \ref{dig 5} commutes.
\begin{figure}[!ht]
	\begin{eqnarray*}
	\begin{tabular}{ccccc}
	& $A$& $\xrightarrow{ \quad  f \quad  }$ & $B$ &\\
	\scriptsize$\alpha$\!\!\!\!\!\!\!\!\!\! & $\Big\downarrow$ & & $\Big\downarrow$ &\!\!\!\!\!\!\!\!\!\! \scriptsize$\beta$ \\
	& $\overline{A}$& $\xrightarrow{\quad h \quad }$ & $\overline{B}$ &
	\end{tabular}
	\end{eqnarray*}
	\caption{\label{dig 5}}
	\end{figure}

From $h\circ \alpha=\beta\circ f\circ \alpha$ it follows that $h=\beta\circ f$, whence $\beta\circ f$ is onto.
Hence $\beta(f(\overline{A}))=\overline{B}$. Also, $\beta(f(\overline{A}))\s \beta(B)\s \overline{B}$, so
$\beta(f(\overline{A}))=\beta(\overline{B})$. Since $\beta$ is one-to-one, $f$ must be onto.
Therefore, $\overline{A}$ is epicomplete.

(ii) Let $(\overline{A},\alpha)$ and $(A',\beta)$ be two epicompletions for $A$. Since $\alpha:A\ra \overline{A}$
is an epiembedding, then by Proposition \ref{2.2}, $\overline{A}$ is an $a$-extension for $f(A)\cong A$.
By (i) and Theorem \ref{2.10.0},  $\overline{A}$ is divisible.  Consider the functor $\Xi$ from Theorem \ref{functor}.
Let $(G(A),u)$ and $(G(\overline{A}),v)$ be the unital $\ell$-groups induced from $MV$-algebras $A$ and $\overline{A}$, respectively. Since there is an epiembedding $A \hookrightarrow \overline{A}$, we have $u=v$.
We have an embedding $\Xi(\alpha):G(A)\ra G(\overline{A})$, and $G(\overline{A})$ is divisible
(see \cite{Dvu2}).
Also, $G(\overline{A})$ is an $a$-extension
for the $\ell$-group $G(A)$ (since there is a one-to-one correspondence between the lattice of ideals of $A$ and
the lattice of convex $\ell$-subgroups of $G(A)$, see \cite[Thm 1.2]{CiTo}), so by \cite[Chap 1, Thm 20]{griffith}, $G(\overline{A})\cong (G(A))^d$.
In a similar way, $G(A')\cong (G(A))^d$. Therefore, $G(\overline{A})\cong G(A')$.
We note that the final isomorphism is an extension for the identity map on $G$, so it preserves the strong units of
$G(\overline{A})$ and $G(A')$, which is a strong unit of $G(A)$, too. Using the functor $\Gamma$, it follows that
$\overline{A}\cong A'$.
\end{proof}

Let $(A;\oplus,',0,1)$ be an $MV$-algebra. By the last proposition if $A$ has an epicompletion, then it is unique up to isomorphic image; this epicompletion is denoted by $(A^e,\alpha)$.

\begin{cor}\label{4.3}
	Let $(A^e,\alpha)$ be an epicompletion for an $MV$-algebra $(A;\oplus,',0,1)$. Then the epicompletion of $A^e$ is equal to $A^e$.
\end{cor}

\begin{proof}
	It follows from Proposition \ref{4.2}(i).
\end{proof}

Now, we try to answer to a question ``whether does an $MV$-algebra have an epicompletion''. First we simply use Theorem \ref{2.13}(i) to show that each linearly ordered $MV$-algebra has an epicompletion.  Then we prove it for
any $MV$-algebra.
For this purpose we try to extend the result of
\cite[Cor 1]{pedersen}. We show that each $\ell$-group has an epicompletion. Then we use this result and we show that
any $MV$-algebra has an epicompletion.

\begin{prop}\label{4.3.0}
Let $(A;\oplus,',0,1)$ be a linearly ordered $MV$-algebra. Then $A$ has an epicompletion.
\end{prop}

\begin{proof}
Let $f:A\ra B$ be an epimorphism.
Since $A$ is a chain, $f(A)$ is also a chain and so, $\mathcal{I}(f(M))$ is a chain. It follows from Proposition \ref{2.2} that
$\mathcal{I}(B)$ is a chain and so $B$ is a linearly ordered $MV$-algebra. Hence by Theorem \ref{2.13}(i),
$B\s (f(A))^d$ (thus $B^d=(f(A))^d$). By \cite[Prop 5]{pedersen} and Theorem \ref{functor}, there is a homomorphism
 $g:A^d\ra B^d$ such that the  diagram in Figure \ref{dig 6}  commutes.
 \begin{figure}[!ht]
	\begin{eqnarray*}
	\begin{tabular}{ccccc}
	& $A$& $\xrightarrow{ \quad  f \quad  }$ & $B$ &\\
	\scriptsize$\s$\!\!\!\!\!\!\!\!\!\!& $\Big\downarrow$ & & $\Big\downarrow$ &\!\!\!\!\!\!\!\!\!\!\!\!\! \scriptsize$\s$ \\
	& $A^d$& $\xrightarrow{\quad g \quad }$ & $B^d$ &
	\end{tabular}
	\end{eqnarray*}
	\caption{\label{dig 6}}
	\end{figure}
$g(A^d)\s B^d$ is a divisible $MV$-algebra containing $B$, so $g$ is onto. Therefore, $A^d$ is an epicompletion for
the $MV$-algebra $A$.
\end{proof}

\begin{rmk}\label{4.4}
Let $G$ and $H$ be two $\ell$-groups and $f:G\ra H$ be an epimorphism. Let $G^d$ and $H^d$ be the divisible hull of $G$ and $H$,
respectively. By \cite[p. 230]{Anderson}, there is a unique extension of $f$ to an epimorphism $\overline{f}:G^d\ra H^d$.
Clearly, if $G$ and $H$ are unital $\ell$-groups and $f$ is a unital $\ell$-group morphism, then so is $\overline{f}$
(for more details see \cite{Anderson} the paragraph after Theorem 2.1 and \cite[Prop 5]{pedersen}). We know that the inclusion maps
$i:G\ra G^d$ and $j:H\ra H^d$ are epimorphisms (by the corollary of \cite[Thm 2.1]{Anderson}),
hence we have the following
commutative diagram (Figure \ref{dig 7}).
\begin{figure}[!ht]
	\begin{eqnarray*}
	\begin{tabular}{ccccc}
	& $G$& $\xrightarrow{ \quad  i \quad  }$ & $G^d$ &\\
	\scriptsize$f$ \!\!\!\!\!\!\!\!\!\!\!\!\!&$\Big\downarrow$ & & $\Big\downarrow$ &\!\!\!\!\!\!\!\!\!\!\!\! \scriptsize$\overline{f}$ \\
	& $H$& $\xrightarrow{\quad j \quad }$ & $H^d$ &
	\end{tabular}
	\end{eqnarray*}
	\caption{\label{dig 7}}
	\end{figure}
	
Since $\overline{f}:G^d\ra H^d$ is an epimorphism and $G^d$ is epicomplete (by \cite[Thm 2.1]{Anderson}), then
$\overline{f}$ is onto and so $G^d$ is an epicompletion for $G$.
\end{rmk}

\begin{thm}\label{4.5}
Any $MV$-algebra has an epicompletion.
\end{thm}

\begin{proof}
Let $(A;\oplus,',0,1)$ be an $MV$-algebra and $f:A\ra B$ be an epimorphism. Then
$\Xi(f):\Xi(A)\ra \Xi(B)$ is a homomorphism of unital $\ell$-groups. By Remark \ref{4.4},
we have the commutative diagram in Figure \ref{dig 1},
where $\overline{\Xi(f)}$ is the unique extension of $\Xi(f)$.
\begin{figure}[!ht]
\begin{eqnarray*}
\begin{tabular}{ccccc}
 &$\Xi(A)$& $\xrightarrow{ \quad  \Xi(f) \quad  }$ & $\Xi(B)$ &\\
\scriptsize$\s$\hspace{-1.3cm} & $\Big\downarrow$ & & $\Big\downarrow$ &\hspace{-1.3cm}\scriptsize$\s$ \\
&$(\Xi(A))^d$& $\xrightarrow{\quad \overline{\Xi(f)} \quad }$ & $(\Xi(B))^d$ &
\end{tabular}
\end{eqnarray*}
\caption{\label{dig 1}}
\end{figure}
Applying the functor $\Gamma$ to diagram in Figure \ref{dig 1}, we get the commutative diagram in Figure \ref{dig 2} on $\mathcal{MV}$.
\begin{figure}[!ht]
\begin{eqnarray*}
\begin{tabular}{ccccc}
&$A$& $\xrightarrow{ \quad \quad f \quad\quad  }$ & $B$&\\
\scriptsize$\s$\hspace{-.7cm} &$\Big\downarrow$ & & $\Big\downarrow$& \hspace{-.7cm}\scriptsize$\s$  \\
&$A^d$& $\xrightarrow{\quad \Gamma(\overline{\Xi(f)}) \quad }$ & $B^d$&
\end{tabular}
\end{eqnarray*}
\caption{\label{dig 2}}
\end{figure}
Set $F:=\Gamma(\overline{\Xi(f)}) $. We claim that $F:A^d\ra B^d$ is an epimorphism. Let $\alpha,\beta:B^d\ra C$ be two
homomorphisms of MV-algebras such that $\alpha\circ F=\beta\circ F$. Then clearly, $\alpha|_{_B}\circ F=\beta|_{_B}\circ F$,
so by the assumption $\alpha|_{_B}=\beta|_{_B}$, which implies that $\Xi(\alpha|_{_B})=\Xi(\beta|_{_B})$. Thus by
\cite{Anderson}, $\overline{\Xi(\alpha|_{_B})}=\overline{\Xi(\beta|_{_B})}$, where
$\overline{\Xi(\alpha|_{_B})},\overline{\Xi(\beta|_{_B})}:(\Xi(B))^d\ra (\Xi(C))^d$ are the unique extensions of $\Xi(\alpha|_{_B})$ and
$\Xi(\beta|_{_B})$, respectively. It can be easily seen that the  diagrams in Figure \ref{dig 3} are commutative.
\begin{figure}[!ht]
\begin{eqnarray*}
\begin{array}{ll}
\begin{tabular}{ccccc}
&$\Xi(B)$& $\xrightarrow{ \quad  \Xi(\alpha|_{_B}) \quad  }$ & $\Xi(C)$ & \\
\scriptsize$\s$\hspace{-1.5cm} &$\Big\downarrow$ & & $\Big\downarrow$ &
\hspace{-1.3cm}\scriptsize$\s$ \\
&$(\Xi(B))^d$& $\xrightarrow{\quad \overline{\Xi(\alpha|_{_B})} \quad }$ & $(\Xi(C))^d$&
\end{tabular}
\end{array}
\hspace{1cm}
\begin{array}{ll}
\begin{tabular}{ccccc}
&$\Xi(B)$& $\xrightarrow{  \Xi(\alpha|_{_B}) \quad  }$ & $\Xi(C)$&\\
\scriptsize$\s$\hspace{-1.4cm} & $\Big\downarrow$ & & $\Big\downarrow$ &
\hspace{-1.4cm} \scriptsize$\s$ \\
&$(\Xi(B))^d$& $\xrightarrow{\quad \Xi(\alpha) \quad }$ & $(\Xi(C))^d$&
\end{tabular}
\end{array}
\end{eqnarray*}
\caption{\label{dig 3}}
\end{figure}
So by the uniqueness of the extension of $\Xi(\alpha|_{B}):\Xi(B)\ra \Xi(C)$ to a map $(\Xi(B))^d\ra (\Xi(C))^d$,  we get that
$\Xi(\alpha)=\overline{\Xi(\alpha|_{_B})}$. In a similar way, $\Xi(\beta)=\overline{\Xi(\beta|_{_B})}$ and so
$\Xi(\alpha)=\Xi(\beta)$. It follows that $\alpha=\Gamma(\Xi(\alpha))=\Gamma(\Xi(\beta))=\beta$. Therefore, $F$ is an
epimorphism. Since $A^d$ is divisible, by Corollary \ref{2.11}, it is epicomplete and so $F$ is onto. That is,
$A^d$ is an epicompletion for $A$. Therefore, any $MV$-algebra has an epicompletion.
\end{proof}


\end{document}